\documentclass{amsart}
\usepackage{amssymb}
\usepackage{graphicx}
\usepackage{amscd}
\usepackage{tikz}
\newtheorem{lemma}{Lemma}[section]
\newtheorem{theorem}[lemma]{Theorem}
\newtheorem{proposition}[lemma]{Proposition}

\theoremstyle{definition}
\newtheorem{definition}[lemma]{Definition}

\theoremstyle{remark}

\newtheorem{remark}[lemma]{Remark}
\newtheorem{remarks}[lemma]{Remarks}

\def\SN{\mathbb N}    %%% N
\def\SZ{\mathbb Z}    %%% Z
\def\SQ{\mathbb Q}    %%% Q
\def\SP{\mathbb P}    %%% R
\def\aa{\mathfrak a}    %%% a
\def\mm{\mathfrak m}    %%% m
\def\pp{\mathfrak p}    %%% p
\def\PP{\mathfrak P}

\def\II{\mathfrak I} 
\def\int{\mbox{\rm{Int}}}             %%% Int

\def\Cc{\mathcal{C}}

\def\Aa{\mathcal{A}}

\def\Ss{\mathcal{S}}
\def\Oo{\mathcal{O}}

\def\Pp{\mathcal{P}}

\def\Ss{\mathcal{S}}

\def\Gal{\mathrm{Gal}}
\def\Max{\mathrm{Max}}
\def\Card{\mathrm{Card}}

\def\be{\begin{equation}}
\def\ee{\end{equation}}

\begin{document}

\title[P\'olya groups of non-galoisian number fields]{From P\'olya fields to P\'olya groups \\(II) Non-galoisian number fields}

\author{Jean-Luc Chabert}
\address{LAMFA (UMR-CNRS 7352), Universit\'e de Picardie, 33 rue Saint Leu, 80039 Amiens, France}
\email{jean-luc.chabert@u-picardie.fr}

%\subjclass[2010]{Primary: 13F20; Secondary: 13B25, 13F05}

%\keywords{}

%%%%%%%%%%%%%%%%%%%%%%%%%%%%%%%%%%%%%%%%
\maketitle

%%%%%%%%%%%%%%%%%%%%%%

\begin{abstract}
The P\'olya group of a number field $K$ is the subgroup of the class group of $K$ generated by the classes of the products of the maximal ideals with same norm. A P\'olya field is a number field whose P\'olya group is trivial. Our purpose is to start with known assertions about P\'olya fields to find results concerning P\'olya groups. In this second paper we describe the P\'olya group of some non-galoisian extensions of $\SQ$.
\end{abstract}

\section{Introduction}

\noindent{\bf Notation}. Let $K$ be a number field with ring of integers $\Oo_K$. For each positive integer $q$, let 
$$ \Pi_q(K)=\prod_{\mm\in\Max(\Oo_K),\,\vert\Oo_K/\mm\vert\,=\,q}\,\mm\,.$$
If $q$ is not the norm of a maximal ideal of $\Oo_K$, then by convention $\Pi_q(K)=\Oo_K$. 

\medskip

The notion of P\'olya fields goes back to P\'olya~\cite{bib:polya} and Ostrowski~\cite{bib:ostrowski} while its formal definition is given by Zantema~\cite{bib:zantema}:

\begin{definition}
A {\em P\'olya field} is a number field $K$ such that all the ideals $\Pi_q(K)$ are principal.
\end{definition}

The notion of P\'olya group of $K$ was introduced later in \cite[\S II.4]{bib:CC} as a measure of the obstruction for $K$ to be a P\'olya field. 

\begin{definition}
The {\em P\'olya group}  of a number field $K$ is the subgroup $\Pp o(K)$ of the class group $\Cc l(K)$ generated by the classes of the ideals $\Pi_q(K)$.	
\end{definition}

Although the P\'olya group of $K$ is also the subgroup of $\Cc l(K)$ generated by the classes of the factorials ideals of $K$ as defined by Bhargava~\cite{bib:bhargava1998} and that, following P\'olya~\cite{bib:polya}, a P\'olya field is in fact a number field $K$ such that the $\Oo_K$-module $\{f\in K[X]\mid f(\Oo_K)\subseteq \Oo_K\}$ formed by the integer-valued polynomials on $K$ admits a basis formed by polynomials of each degree, we will not use these properties.
Of course, the field $K$ is a P\'olya field if and only if its P\'olya group is trivial. So that, every result about P\'olya group has a corollary which is a result about P\'olya fields. 
Here we are going to make a reverse lecture of this fact: we make the conjecture that an assertion about P\'olya fields is the evidence of a statement about P\'olya groups.

Almost all the results about either P\'olya fields or P\'olya groups are obtained in the Galois case (see \cite{bib:chabertI}) since this is the easier case thanks to the following remark due to Ostrowski~\cite{bib:ostrowski}:
if the extension $K/\SQ$ is galoisian and if $p$ is not ramified in $K/\SQ$, then the ideals $\Pi_{p^f}(K)$ are principal. Thus, in the Galois case, we just have to consider the finitely many ramified primes. On the other hand, in the non-galoisian case, we have to consider {\em a priori} the ideals $\Pi_{p^f}(K)$ for all primes $p$.

Among the few known results in the non-galoisian case, we have that the Hilbert class field of every number field is a P\'olya field~\cite[Corollary 3.2]{bib:leriche3}. But the most important result about P\'olya fields that we want to translate into a result about P\'olya groups is the following one due to Zantema:

\begin{theorem}\cite[Thm 6.9]{bib:zantema}\label{th:69}
Let $K$ be a number field of degree $n\geq 3$ and let $G$ be the Galois group of the normal closure of $K$ over $\SQ$. If $G$ is either isomorphic to the symmetric group $\Ss_n$ where $n\not=4$, or to the alternating group $\Aa_n$ where $n\not=3, 5$, or is a Frobenius group, then the following assertions are equivalent:
\begin{enumerate}
\item all the ideals $\Pi_{p^f}(K)$ where $p$ is not ramified are principal,
\item $K$ is a P\'olya field,
\item $h_K=1$.
\end{enumerate}
\end{theorem}

In other words, under the previous hypotheses, denoting by $\Pp o(K)_{nr}$ the subgroup of $\Pp o(K)$ generated by the classes of the ideals $\Pi_{p^f}(K)$ where $p$ is a not ramified, one has the equivalences:
$$\Pp o(K)_{nr}=\{1\}\Leftrightarrow \Pp o(K)=\{1\}\Leftrightarrow \Cc l(K)=\{1\}\,.$$
Note that $\Pp o(K)_{nr}$ is always trivial when $K/\SQ$ is galoisian. Motivated by what we did in the Galois case~\cite{bib:chabertI}, we formulate the following unreasonable conjecture:

\medskip

\noindent{\bf Conjecture}~\cite{bib:chabert2018}.
{\em Each of the hypotheses of Theorem~\ref{th:69} implies the following  equalities:}
$$\Pp o(K)_{nr}=\Pp o(K)=\Cc l(K)\,.$$
The aim of this paper is to give a proof of this conjecture (cf. Theorem~\ref{th:47x}). At least at the beginning of this proof, the ideas that we will  use follow closely that given by Zantema for Theorem~\ref{th:69}

%%%%%%%%%%%%%%%%%%%%%%%%%%%

\section{The Artin symbol $\left(\frac{H_K/K}{\Pi_{p^f}(K)}\right)$ for a non-ramified $p$}

%%%%%%%%%%%%%%%%%%%%%%

\noindent{\bf Notations}. In the sequel, $K$ is a number field of degree $n\geq 3$. We denote by $\Oo_K$ its ring of integers, $\II_K$ the group of its nonzero fractionary ideals, $\Pp(K)$ the subgroup of nonzero fractionary principal ideals, $\Cc l(K)=\II_K/\Pp_K$ its class group, $h_K$ its class number, $\Pi_q(K)$ the product of its prime ideals with norm $q$. Let $H_K$ be the Hilbert class field of $K$, that is, the largest non-ramified abelian extension of $K$, and let $N_H$ be the normal closure of $H_K$ over $\SQ$. Finally, we denote the Galois groups in the following way:  $G=\Gal(N_H/\SQ),$ $H=\Gal(N_H/K),$ and $H_1=\Gal(N_H/H_K)$.

\medskip

There exists an isomorphism of groups  (see for instance~\cite[I \S 8]{bib:serre})
$$\gamma_K:\Cc l(K)\to \Gal(H_K/K)=H/H_1\,.$$
To study the group $\Pp o(K)$, and more specifically the subgroup $\Pp o(K)_{nr}$, we will use this isomorphism. Thus, let us describe $\gamma_K$.
Recall that, if $L/K$ is a galoisian extension and $\PP$ is a prime of $L$ non ramified in the extension $L/K$ and lying over $\pp=\PP\cap K$, the {\em Frobenius} of $\PP$ in the extension $L/K$ is the element $\sigma$ of $\Gal(L/K)$ characterized by:
$$ \forall a\in \Oo_L\quad\sigma(a)\equiv a^q\pmod{\PP}\quad \textrm{where } q=\textrm{Card}(\Oo_K/\pp) \,.$$ 
This Frobenius, that we denote by $(\PP,L/K)$, generates the decomposition group of $\PP$ in the extension $L/K$, thus its order is $f_{\PP}(L/K)=[\Oo_L/\PP:\Oo_K/\pp]$. 

If the extension $L/K$ is abelian, the Frobenius of $\PP$ depends only on $\pp=\PP\cap K$, it is called the {\em Artin symbol} of $\pp$ and is denoted by $\left(\frac{L/K}{\pp}\right)$. By Chebotarev's theorem, every element of $\Gal(L/K)$ is the Artin symbol of a non-ramified prime $\pp$. More generally, the Artin symbol of an ideal $\aa=\prod_i\pp_i^{\alpha_i}$ of $\Oo_K$ which is not contained in any ramified prime ideal of $\Oo_K$ is defined by linearity: $\left(\frac{L/K}{\aa}\right) = \prod_i\left(\frac{L/K}{\pp_i}\right)^{\alpha_i}$. 

In the particular case of the extension $H_K/K$, this surjective morphism of groups $\aa\in\II_K\mapsto\left(\frac{H_K/K}{\aa}\right)\in\Gal(H_K/K)$  induces the isomorphism $\gamma_K:$
$$\label{eq:13}  \gamma_K:\overline{\aa}\in\Cc l(K)\mapsto \left(\frac{H_K/K}{\aa}\right)\in \mathrm{Gal}(H_K/K)\,. $$

\smallskip

Let $\Omega$ be the set formed by the right cosets of $H$ in $G$. The group $G$ acts transitively on ${\Omega}$. As a permutation of ${\Omega}$, an element $g\in G$ is a product of $t$ disjoint cycles of orders $f_i$ ($1\leq i\leq t$) (fixed points correspond to cycles of order one). For each $i$, let $s_i\in G$ be such that $Hs_i\in{\Omega}$ belongs to the orbit of the $i$-th cycle. Thus, ${\Omega}=\{Hs_ig^{k}\mid 1\leq i\leq t, 0\leq k<f_i\}$

\begin{proposition}\cite[Prop. 6.2]{bib:zantema}\label{th:21a}
Assume that $p$  is not ramified in the extension $K/\SQ$. Let $\PP$ be a prime ideal of $N_H$ lying over $p$ and consider its Frobenius $g=(\PP,N_H/\SQ)$. With the previous notation ($t$, $f_i$ and $s_i$ for $g$), one has:
\be\label{eq:10}\left(\frac{H_K/K}{\Pi_{p^f}(K)}\right)=\prod_{\{i \mid f_i=f\}} s_i\,g^f\, s_i^{-1} \,\vert_{H_K}\,.\ee	
\end{proposition}

\begin{proof}
Note that the fact that $p$ is not ramified in $K/\SQ$ implies that $\PP$ is not ramified in $N_H/\SQ$. By definition,
$$ \left(\frac{H_K/K}{\Pi_{p^f}(K)}\right)= \prod_{\pp\in\Max(\Oo_K),\, N(\pp)=p^f}\left(\frac{H_K/K}{\pp}\right)\,. $$ 
By Lemma~\ref{th:42a} below,
$$\label{eq:8} \left(\frac{H_K/K}{\Pi_{p^f}(K)}\right)= \prod_{\{i \mid f_i=f\}}\left(\frac{H_K/K}{s_i(\PP)\cap K}\right)\,. $$ 
An Artin symbol is a Frobenius:
$$\left(\frac{H_K/K}{s_i(\PP)\cap K}\right)= (s_i(\PP)\cap H_K,H_K/K) =(s_i(\PP),N_H/K)_{\vert H_K}\,.$$
It follows from the properties of the Frobenius that 
$$(s_i(\PP),N_H/K)=(s_i(\PP),N_H/\SQ)^{f_i}=\left(s_i\,(\PP,N_H/\SQ)\, s_i^{-1}\right)^{f_i}=s_i\,g^{f_i}\,s_i^{-1}\,.$$
Consequently,

$\quad\quad\quad\quad\quad\quad\quad\quad\quad\quad\quad \left(\frac{H_K/K}{s_i(\PP)\cap K}\right) = s_i\,g^{f_i}\, s_i^{-1} \vert_{H_K}\,.$
\end{proof}

\begin{remarks}\label{rem:22r}
With the hypotheses and notation of Proposition~\ref{th:21a}. 

(i) Note that, in Formula~(\ref{eq:10}),

 -- for every $i$, $s_ig^{f_i}s_i^{-1}\in H$, 

-- to consider the restriction to $H_K$ is equivalent to consider the class modulo~$H_1$,

-- as the group $H/H_1$ is abelian, the product in (\ref{eq:10}) does not depend on the order.

(ii) By Formula (\ref{eq:10}) itself, the class mod $H_1$ of the product $\prod_{\{i \mid f_i=f\}} s_i\,g^f\, s_i^{-1}$ does not depend on the choice of the $s_i$'s~\cite[p. 176]{bib:zantema}.

(iii) It follows from the equality $p\Oo_K=\prod_{1\leq i\leq t}\Pi_{p^{f_i}}(K)$ 	that: $$1=\gamma_K(\overline{p\Oo_K})=\left(\frac{H_K/K}{p\Oo_K}\right)=\prod_{1\leq i\leq t}s_i\,g^f s_i^{-1} \,\vert_{H_K}\,,$$
which means that \cite[Prop. 6.2]{bib:zantema}
$$\prod_{1\leq i\leq t}s_i\,g^f s_i^{-1} \in H_1\,.$$
\end{remarks}

\begin{lemma}\cite[III, Prop. 2.8]{bib:janusz}\label{th:42a}
With the hypotheses and notation of Proposition~\ref{th:21a}:
\begin{enumerate}
\item the prime ideals of $K$ lying over $p$ are the $t$ ideals $s_i(\PP)\cap K\;(1\leq i\leq t)$,
\item the residual degree of $s_i(\PP)\cap K$ in the extension $K/\SQ$ is $f_i$.
\end{enumerate}
\end{lemma}

%%%%%%%%%%%%%%%

\section{The set $T$ formed by the elements of $H$ with one fixed point}

%%%%%%%%%%%%%%%%
Let $T$ be the set formed by the elements of $H$ with only one fixed point, namely the class $H$, in their action on $\Omega:$
$$T=\{h\in H\mid \forall s\in G\setminus H\;\;Hsh\not=Hs\, \}\,. $$

\begin{lemma}\label{th:61A} In the group $H/H_1$, one has the containment:
$$T\;\mathrm{mod}\; H_1\;\subseteq \;\left\{\left(\,\frac{H_K/K}{\Pi_p(K)}\right)\,\Big\vert\, p\in\SP \textrm{ non-ramified in } K\,\right\}\,.$$
\end{lemma}

\begin{proof}
Let $g\in T\subseteq H$. By Cheborev's density theorem, $g$ is the Frobenius of a prime $\PP$ of $N_H$ lying over a prime number $p$ which is not ramified in $K$. Following formula (\ref{eq:10}), $\left(\,\frac{H_K/K}{\Pi_p(K)}\right)=\prod_{\{i\mid f_i=1\}}s_igs_i^{-1} \textrm{ mod } H_1$. As $H$ is the unique fixed point of $g$ in its action on $\Omega$, there is exactly one $f_i$ which is equal to 1 and we may choose $s_i=1$. Consequently, $\left(\,\frac{H_K/K}{\Pi_p(K)}\right)=g \textrm{ mod } H_1$.
\end{proof}

Recall that
$$\Pp o(K)=\langle\,\overline{\Pi_q(K)}\mid q\in\SN^*\,\rangle$$
and
$$\Pp o(K)_{nr}=\langle\,\overline{\Pi_{p^f}(K)}\mid f\in\SN^*, p\in\SP \textrm{ non-ramified in } K\,\rangle\,.$$
Let also:
$$\Pp o(K)_{nr1}=\langle\,\overline{\Pi_{p}(K)}\mid p\in\SP \textrm{ non-ramified in } K\,\rangle\,.$$

\begin{proposition}\label{th:66A}
If the following equality holds
$$H=\;\langle\, T,H_1 \,\rangle\,,$$
then 
$$ \Cc l(K)=\Pp o(K) =\Pp o(K)_{nr}= \Pp o(K)_{nr1}\,.$$
\end{proposition}

\begin{proof}
We have to prove the containment  $\Cc l(K)\subseteq \Pp o(K)_{nr1}$\,. The hypothesis means that $ H/H_1= \;\langle \,T \;\mathrm{mod}\, H_1 \,\rangle$. 
Applying the isomorphism $\gamma^{-1}_K$ to both sides, the image of the left hand side is
$ \gamma_K^{-1}(H/H_1)=\Cc l(K)$
while, by lemma~\ref{th:61A}, the image of the right hand side $\gamma^{-1}_K(\,\langle T \;\mathrm{mod}\,H_1\,\rangle)$ is contained in $\Pp o(K)_{nr1}\,.$
\end{proof}

Thus, we obtain the conclusion of the conjecture, but how can we know that the hypothesis is satisfied? In some cases, in particular if the hypotheses of Theorem~\ref{th:69} are satisfied, we may replace the normal closure of $H_K$ by those of $K$.

\medskip

Let $N_K$ be the normal closure of $K$ over $\SQ$. We let $G_0=\Gal(N_K/\SQ)$,  $H_0=\Gal(N_K/K),$ and $H_2=\Gal(N_H/N_K)$. 
Of course, $H_2$ is the largest normal subgroup of $G$ contained in $H$. Thus, 
$$ H_2=\cap_{g\in G}\;gHg^{-1}\,.
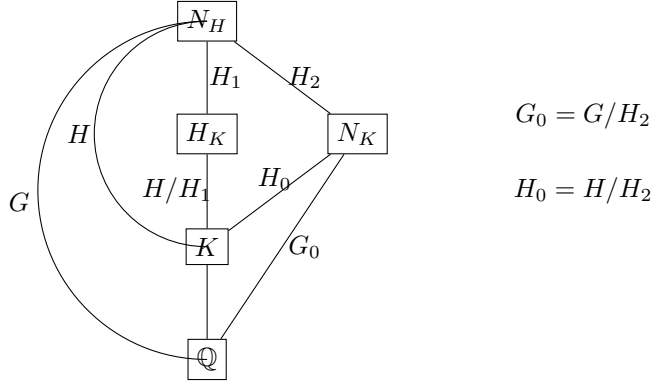
$$ 

\medskip

\begin{tikzpicture}

\node[draw] (Q) at (5,0) {$\SQ$};
\node[draw] (K) at (5,1.5) {$K$};
\node[draw] (HK) at (5,3) {$H_K$};
\node[draw] (NH) at (5,4.5) {$N_H$};
\node[draw] (NK) at (7,3) {$N_K$};
%\node[draw] (L) at (7,7) {$L$};

\node (H1) at (5.25,3.75) {$H_1$};
\node (HH1) at (4.6,2.25) {$H/H_1$};
\node (H) at (3.3,3) {$H$};
\node (G) at (2.5,2.1) {$G$};
\node (H_2) at (6.3,3.75) {$H_2$};
\node (G_0) at (6.3,1.5) {$G_0$};
\node (H_0) at (5.9,2.4) {$H_0$};

\node (a) at (10,2.25) {$H_0=H/H_2$};
\node (b) at (10,3.25) {$G_0=G/H_2$};
%\node[draw , shape=circle] (T) at (6,11) {$<T>$};
%\node[draw , shape=circle] (H2) at (8,12) {$H_2$};
%\node[draw , shape=circle] (TH2) at (9,8) {$<T>/H_2$};
%\node[draw , shape=circle] (HH2) at (9,5) {$H/H_2$};
%\node[draw , shape=circle] (GH2) at (9,2) {$G/H_2$};

\draw (Q) edge[] (K);
\draw (K) edge[] (HK);
\draw (HK) edge[] (NH);
\draw (NK) edge[] (NH);
\draw (NK) edge[] (K);
\draw (NK) edge[] (Q);
%\draw (K) edge[] (L);
%\draw (L) edge[] (NK);
%\draw (L) edge[] (NH);

\draw (5,4.5) arc (90:270:2.25);
\draw (5,4.5) arc (90:270:1.5);
%\draw (5,4) arc (250:360:5);
%\draw (5,0) arc (270:360:7);

\end{tikzpicture}

\medskip

Analogously to that we considered with respect to $G$ and $H$, we introduce the set $\Omega_0$ formed by the right cosets of $H_0$ in $G_0$, and the set $T_0$ formed by the elements of $H_0$  with only one fixed point in their action on $\Omega_0$. Moreover, we consider the canonical surjection 
$\pi:G\to G/H_2=G_0.$

\begin{lemma}\label{th:719}
For every $h\in H$, $h\in T$ if and only if $\pi(h)\in T_0$. Moreover, if $T\not=\emptyset$, then $H_2\subseteq  \langle \,T\,\rangle \,.$ \end{lemma}

\begin{proof}
Let $h\in H$. Then, $h\notin T $
$\Leftrightarrow \exists s\in G\setminus H$ s.t. $Hsh^{-1}=Hs$
$\Leftrightarrow\exists s\in G\setminus H$ s.t. $shs^{-1}\in H $
$\Leftrightarrow \exists s_0\in G_0\setminus H_0$ s.t.  $s_0\pi(h)s_0^{-1}\in H_0 $
$\Leftrightarrow \pi(h)\notin T_0$.

Assume now that $T\not=\emptyset$ and fix some $t\in T$. Consider $h_2\in H_2$. 
As $t\in T$, one has $\pi(th_2)=\pi(t)\in T_0$, and hence, $th_2\in T$. Finally, $h_2=t^{-1}\times th_2\in\langle\,T\,\rangle$.
\end{proof}

Denoting by $H_0'$ the derived subgroup of $H_0$, we have:

\begin{lemma}\label{th:72A}
If $T_0\not=\emptyset$ and $H_0=\;\langle\, T_0,H'_0\,\rangle $\,, then $H=\;\langle \, T,H_1\, \rangle $\,.	
\end{lemma}
\begin{proof}
It follows from Lemma~\ref{th:719} that $T_0\not=\emptyset$ implies that $T\not=\emptyset$, and hence, that $H_2\subseteq \langle\, T\, \rangle\,$. Clearly, $\pi(H')=H'_0$. Consequently,
$$\pi(\langle\,T,H'\,\rangle)=\langle\,\pi(T),\pi(H')\,\rangle =\langle\, T_0,H_0'\,\rangle =H_0\,.$$
Since $H_2\subseteq\langle\,T\,\rangle$, one has $H=\langle\,T,H'\,\rangle\,,$ and since $H/H_1$ is abelian, one has $H'\subseteq H_1$. Finally, 
$$H=\langle\,T,H'\,\rangle\subseteq \langle \, T,H_1\, \rangle\subseteq H\,.$$
\end{proof}

%%%%%%%%%%%%%%%%%

\section{Conclusion}

\begin{theorem}\label{th:720}
Let $K$ be a number field of degree $n\geq 3$. Let $N_K$ be the normal closure of $K$ over $\SQ$. Let $G=\Gal(N_K/\SQ)$ and $H=\Gal(N_K/K)$. Let ${\Omega}$ be the set formed by the right cosets of $H$ in $G$ and let $T$ be the set formed by the elements of $H$ which, in their action on ${\Omega}$, have only $H$ as fixed point.
Assume that both following conditions are satisfied:
\be\label{eq:2B} T \not=\emptyset \quad\textrm{ and } \quad H=\;\langle\, T, H'\,\rangle \,.\ee
Then, the class group $\Cc l(K)$ of $K$ is generated by the classes of the ideals $\Pi_p(K)$ where $p\in\SP$ is not ramified in $K.$	
\end{theorem}

\begin{proof}
This is a straightforward consequence of Lemma~\ref{th:72A} and  Proposition~\ref{th:66A}.
\end{proof}

\begin{remark}\cite[Remark 6.5]{bib:zantema}\label{rem:42}
If the action of $G$ on ${\Omega}$ is 2-transitive, then the action of $H$ on $\Omega\setminus\{H\}$ is transitive. Since $\Card(\Omega\setminus\{H\})=n-1\geq 2$, it follows from Burnside's formula that $T\not=\emptyset$.	
\end{remark}

In the following applications of Theorem~\ref{th:720}, $G:=\Gal(N_K/\SQ)$ is considered as isomorphic to a subgroup of the symmetric group $\Ss_n$ since, if $K$ is generated by $x$, the elements of $G$ are permutations of the $n$ conjugates $x=x_1,x_2,\ldots,x_n$ of $x$ over $\SQ$. Moreover, $H:=\Gal(N_K/K)$ is then the subgroup of $G$ which fixes $x$.

\begin{lemma}\label{th:718}
Conditions (\ref{eq:2B}) of Theorem~\ref{th:720} are satisfied if $G\simeq S_n$ with $n=3$ or $n\geq 5$, or if $G\simeq \Aa_n$ with $n=4$ or $n\geq 6$.
\end{lemma}

\begin{proof}
(a) Assume that $G:=\Gal(N_K/\SQ)\simeq\Ss_n$. Then, $H:=\Gal(N_K/K)\simeq \Ss_{n-1}$, $H'\simeq\Aa_{n-1}$\,, and $\Omega=\{\Ss_{n-1}\sigma_i\mid 1\leq i\leq n\}$ where $\sigma_i\in\Ss_n$ satisfies $\sigma_i(i)=n$. As $H/H'\simeq \Ss_{n-1}/\Aa_{n-1}\simeq\SZ/2\SZ$, both conditions  (\ref{eq:2B}) will surely be satisfied if there exists some $\tau\in\Ss_{n-1}\setminus \Aa_{n-1}$ which does not fix any symbol. We obtain such a $\tau$ by considering any cycle of length $n-1\geq 2$ in $\Ss_{n-1}$ if $n$ is odd (which leads to the constraint $n\geq 3$), and, if $n$ is even, by considering the product of two disjoint cycles of length $r\geq 2$ and $s\geq 2$ respectively such that $r+s=n-1$ (which leads to the constraint  $n\geq 5$). 

\smallskip

(b) Assume now that $G:=\Gal(N_K/\SQ)\simeq\Aa_n$. Then, $H:=\Gal(N_K/K)\simeq \Aa_{n-1}$. $\Aa_{n-1}$ acts transitively on the $n$ right cosets of $\Aa_n$ modulo $\Aa_{n-1}$ as soon as $n\geq 4$. It then follows from Remark~\ref{rem:42} that $T\not=\emptyset$. Moreover, as $H'\simeq\Aa_{n-1}$ for $n-1\geq 5$, conditions~(\ref{eq:2B}) are satisfied for $n\geq 6$. And also for $n=4$ since $\{1\}\subsetneq T\subseteq \Aa_3$.
	\end{proof}

Frobenius groups lead also to similar conclusions. Recall :

\begin{definition}
A {\em Frobenius group} is a permutation group $G$ which acts transitively on a finite set $X$ such that every $g\in G\setminus\{1\}$ does not fix more than one point and that there exists at least one $g\in G\setminus\{1\}$ which fix an $x\in X$.	
\end{definition}

\begin{lemma}
Conditions (\ref{eq:2B}) of Theorem~\ref{th:720} are satisfied if $G$ is a Frobenius group. 
\end{lemma}

\begin{proof}
By transitivity the stabilizer of every element is not trivial, and hence, $H\not=\{1\}$. It follows then from the definition that $T=H\setminus\{1\}\not=\emptyset$.
\end{proof}

Now we can conclude:

\begin{theorem}\label{th:47x}
Let $K$ be a number field of degree $n\geq 3$. If the Galois group of the normal closure of $K$ is either isomorphic to $\Ss_n \; (n\not=4)$ or to $\Aa_n \; (n\not=3,5)$, or is a Frobenius group, then
$$\Cc l(K)= \Pp o(K) = \Pp o(K)_{nr} =\Pp o(K)_{nr1}\,. $$	
\end{theorem}

For instance, the P\'olya group $\Pp o(K)$ of a non-galoisian cubic number field $K$ is always equal to the class group $\Cc l(K)$ of $K$.	
%%%%%%%%%%%%%%%%%%%

\end{document}